\documentclass[bibliography=totoc,fontsize=10pt]{scrartcl}

\usepackage[margin=1in]{geometry}  
\usepackage{amssymb}
\usepackage{amsmath}               
\usepackage{amsfonts}              
\usepackage{amsthm}                
\usepackage{paralist}
\usepackage[usenames,dvipsnames]{color} 
\usepackage[british]{babel}
\usepackage{cite}
\usepackage[affil-it]{authblk}
\usepackage{setspace}                
\usepackage{hyperref}
\usepackage[utf8]{inputenc}       
\usepackage{bbding}  
\usepackage{nicefrac}

\newtheorem{Thm}{Theorem}
\newtheorem{thm}{Theorem}[section]
\newtheorem{lem}[thm]{Lemma}
\newtheorem{prop}[thm]{Proposition}

\newtheorem{defi}[thm]{Definition}
\newtheorem{rem}[thm]{Remark}

\newtheorem{innercustomgeneric}{\customgenericname}
\providecommand{\customgenericname}{}
\newcommand{\newcustomtheorem}[2]{%
  \newenvironment{#1}[1]
  {%
   \renewcommand\customgenericname{#2}%
   \renewcommand\theinnercustomgeneric{##1}%
   \innercustomgeneric
  }
  {\endinnercustomgeneric}
}

\newcustomtheorem{RefProp}{Proposition}
\newcustomtheorem{RefLem}{Lemma}

\newenvironment{proofof}[1]{\renewcommand{\proofname}{Proof of #1}\proof}{\endproof}

\DeclareMathAlphabet\mathbb{U}{fplmbb}{m}{n}

\newcommand{\RR}{\mathbb{R}}     

\newcommand{\ZZ}{\mathbb{Z}}     
\newcommand{\NN}{\mathbb{N}}

\newcommand{\defeq}{\mathrel{\mathop{\raisebox{1.1pt}{\scriptsize$:$}}}=}

\newcommand\opna{\operatorname}
\newcommand\mf{\mathfrak}
\newcommand\mc{\mathcal}

\newcommand\Lip{\operatorname{Lip}}
\newcommand\setI{\opna{\textup{\textbf{I}}}}
\newcommand\Mass{\opna{\textup{\textbf{M}}}}
\newcommand\FillVol{\opna{\textup{FillVol}}}

\addtokomafont{section}{\normalsize \rmfamily\scshape}
\addtokomafont{subsection}{\small \rmfamily\scshape}
\addtokomafont{subsubsection}{\small \rmfamily\scshape}

\setkomafont{sectioning}{\normalcolor\rmfamily\scshape}

\begin{document}

\title{\Large Non-Euclidean isoperimetric inequalities\\ for nilpotent Lie groups}
\date{}
\author{\large Moritz Gruber \footnote{The author is supported by the German Research Foundation (DFG) grant GR 5203/1-1}}
\maketitle

\vspace*{-10mm}

\begin{abstract}\noindent\small
\textbf{Abstract.} This article treats isoperimetric inequalities for integral currents in the setting of stratified nilpotent Lie groups equipped with left-invariant Riemannian metrics. We prove that for each such group there is a dimension in which no Euclidean isoperimetric inequality is admitted, while in all smaller dimensions strictly Euclidean isoperimetric inequalities are satisfied.
\end{abstract}

%
%
\pagenumbering{arabic}

\section{Introduction}\label{S1}

Isoperimetric inequalities have been and are object of geometric research from the very beginning of mathematics some thousand years ago until today (and hopefully will be so in the future). Starting with \emph{Dido's Problem} of enclosing an as large as possible area of the plane by a rope of a given length, the question has extended to more general settings. We now ask  in a metric space $(X,d_X)$ for the maximal needed $(k+1)$-dimensional volume to fill a $k$-dimensional boundary of given surface measure. 

For the case of a simply connected (stratified) nilpotent Lie group $G$ with a left-invariant Riemannian length metric $d_R$, Gromov has drawn a conjectural picture of the isoperimetric behaviour (see \cite{GGT}). Up to a dimension $k_o$ the space $(G,d_R)$ should admit strictly Euclidean isoperimetric inequalities, i.e. the maximal needed volume of a $(k+1)$-dimensional filling, $k < k_o$, grows polynomial of degree $\nicefrac{(k+1)}{k}$ in the surface measure of its boundary. Then in dimension $k_o$ a non-Euclidean isoperimetric inequality is expected, which by Gromov's heuristic arguments (mainly based on observations on the Heisenberg groups) has a polynomial growth of degree $\nicefrac{(k_o+2)}{k_o}$ as lower bound. Since then there has been a steady progress, but most works either focus on the Euclidean upper bounds in low dimensions  (e.g. \cite{Pittet95},\cite{Young1},\cite{Gruber1}) or on groups of Heisenberg type (e.g. \cite{Burillo},\cite{Pittet},\cite{YoungII},\cite{Gruber2}). Of a different spirit is the considerable result of Wenger in \cite{Wenger11} which necessitates some adjustment on the conjecture. It gives examples of stratified nilpotent Lie groups with a super-quadratic lower bound for the isoperimetric inequality in dimension 1 with at the same time an upper bound proportional to $\ell^2\log(\ell)$. This means these groups admit in dimension $k_o=1$ no Euclidean isoperimetric inequality, but the upper bound is much smaller than predicted by Gromov. The there used integral $m$-currents are generalisations of differentiable (sub-) manifolds. The solution of \emph{Plateau's Problem} by Federer and Fleming in \cite{FF60} legitimated them as the right objects to examine isoperimetric problems.

In the present article we prove the conjecture discussed above. 
The main result is Theorem \ref{Thm}, which states that every stratified nilpotent Lie group with a left-invariant Riemannian metric satisfies strictly \mbox{Euclidean} isoperimetric inequalities up to a dimension $k_o$ in which it does not admit an Euclidean isoperimetric inequality anymore.

\begin{Thm}\label{Thm}
For every stratified nilpotent Lie group $G$ of step $d\ge 2$, equipped with a left-invariant Riemannian length metric $d_R$, there is a dimension $k_o \in \NN$ such that $(G,d_R)$ does not admit an Euclidean isoperimetric inequality for $\setI_{k_o}(G,d_R)$ and satisfies strictly Euclidean isoperimetric inequalities for $\setI_k(G,d_R)$ for all $k <k_o$.
\end{Thm}

We'd like to draw the reader's attention to the guaranteed non-Euclidean isoperimetric inequality in dimension $k_o$. By Wenger's result \cite{Wenger11} this presumably is the best possible general statement. 
To prove Theorem \ref{Thm}, we show that Euclidean isoperimetric inequalities imply the existence of closed horizontal boundaries.

\begin{RefLem}{\ref{LemApprox}}
Let $(G,d_R)$ be a stratified nilpotent Lie group and let $m\in \NN$, $m< \dim(G)$, be such that $(G,d_R)$  admits for all $k\in \{1,...,m\}$  an Euclidean isoperimetric inequality for $\setI_k(G,d_R)$.
Then there is a closed horizontal current $S \in \setI_m(G,d_R)$ with $\FillVol(S)>0$.
\end{RefLem}

Using this, we will prove that the first not strictly Euclidean isoperimetric inequality cannot have an Euclidean upper bound.

\begin{RefProp}{\ref{prop}}
Let $(G,d_R)$ be a stratified nilpotent Lie group and let $m\in \NN$ be the lowest dimension in which $(G,d_R)$ does not satisfy a strictly Euclidean isoperimetric inequality for integral currents. Then $(G,d_R)$ does not admit an Euclidean isoperimetric inequality for $\setI_{m}(G,d_R)$.
\end{RefProp}

Once we have these results, Theorem \ref{Thm} follows directly:

\begin{proofof}{Theorem \ref{Thm}}
Let $V_1$ be the horizontal distribution of $G$. By Lemma \ref{LemApprox} we know $(G,d_R)$ cannot admit Euclidean isoperimetric inequalities for $\setI_k(G,d_R)$ for all $k \le \dim (V_1)+1$ as there are no non-trivial horizontal currents of dimension greater than $\dim (V_1)$.
So there is a dimension $m \in \{1,...,\dim (V_1)+1\}$ which is the lowest dimension in which $(G,d_R)$ does not satisfy a strictly Euclidean isoperimetric inequality for integral currents. By Proposition \ref{prop} $(G,d_R)$ does not admit an Euclidean isoperimetric inequality for $\setI_{m}(G,d_R)$  and Theorem \ref{Thm} follows with $k_o=m$.
\end{proofof}

Indeed, the maximal dimension of non-trivial horizontal currents is strictly smaller than the dimension of the horizontal distribution $V_1$ as for the set of horizontal currents holds $\setI_k(G,d_C)=\{0\}$ if there isn't a $k$-dimensional Abelian subalgebra of $\mf g$ contained in $V_1$ (see \cite{Magnani}).


\section{Definitions and notation}

\subsection{Nilpotent Lie groups}

\begin{defi}[Stratified nilpotent Lie group]
A nilpotent Lie group  $G$ of step $d$ with Lie algebra $\mf g$ is \emph{stratified} if it is simply connected and there is a grading
$\mf g = V_1 \oplus V_2 \oplus ... \oplus V_d$
with $[V_1, V_j]=V_{j+1}$.
\end{defi}

Throughout the whole paper $G$ will denote a stratified nilpotent Lie group of step $d\ge 2$ and with Lie algebra $\mf g= V_1 \oplus V_2 \oplus ... \oplus V_d$. This especially includes all simply connected nilpotent Lie groups of step $2$. The first layer $V_1$ is called the \emph{horizontal} distribution. We equip $G$ with a left-invariant Riemannian metric $\langle\cdot,\cdot\rangle$. As all such metrics are Lipschitz-equivalent and Lipschitz equivalences only affect the constants in our isoperimetric inequalities, we can assume $V_i\perp V_j$ for all $i \ne j$ with respect to this metric. We denote by $d_R$ the corresponding length metric and the \emph{Carnot-Carath\'eodory metric}, i.e. the corresponding length metric of almost everywhere $V_1$-tangent curves, we denote by 
$$d_C(x,y) \defeq \inf\{\opna{length}_R(\gamma_{xy}) \mid \gamma_{xy} \text{ curve from $x$ to $y$ with } \dot \gamma_{xy}(t) \in dL_{\gamma_{xy}(t)} V_1 \text{ a.e.}\} \quad \forall x,y \in G$$
where $L_g:G \to G, x \mapsto gx$ denotes the left-multiplication by $g\in G$. Such a group $G$ has a family of automorphisms $(s_t:G \to G)_{t>0}$ with $\opna{d}_e s_t: \mf g \to \mf g, (v_1,...,v_d) \mapsto (tv_1, t^2v_2,...,t^dv_d)$. The two metrics of interest have the following behaviour under these maps:
$$d_R(s_t(x),s_t(y)) \begin{cases} \le t \cdot d_R(x,y) \text{ if } t <1 \\  \ge t \cdot  d_R(x,y) \text{ if } t \ge1 \end{cases} \text{ and }\ d_C(s_t(x),s_t(y))=t \cdot  d_C(x,y) \quad \forall x,y \in G.$$
We will denote the Hausdorff distance with respect to these metrics by $d_{\mc H,d_R}$ and $d_{\mc H,d_C}$, respectively. Further we will denote for $U\subset G$ by $\dim_{\mc H,d_R}(U)$ and $\dim_{\mc H,d_C}(U)$ the Hausdorff dimensions corresponding to the metrics $d_R$ and $d_C$.
%
%

\subsection{Integral currents}

 We will give only a very brief definition of integral currents and the most important terms. For an exhaustive introduction see \cite{FF60} and \cite{AmbrosioKirchheim}.

Let $(X,d_X)$ be a complete metric space and $k \in \NN_0$. Let $\mc D^k(X)$ be the $\RR$-vector space of $(k+1)$-tuples $(f,\pi_1,...,\pi_k)$ of Lipschitz functions on $(X,d_X)$ with first element $f$ always bounded. A $k$-dimensional \emph{metric functional} on $(X,d_X)$ is a sub-additive and positive $1$-homogeneous function $T:\mc D^k(X) \to \RR$. The \emph{boundary} of a $k$-dimensional metric functional $T$ is the $(k-1)$-dimensional metric functional $\partial T$ defined by $\partial T(f,\pi_1,...,\pi_{k-1})=T(1,f,\pi_1,...,\pi_{k-1})$ for all $(f,\pi_1,...,\pi_{k-1}) \in \mc D^{k-1}(X)$. A $k$-dimensional metric functional $T$ is called of \emph{finite mass} if there is a finite Borel measure $\mu$ on $X$ such that
$$|T(f,\pi_1,...,\pi_k)| \le \prod_{i=1}^k \Lip(\pi_i) \int_X |f|d\mu \qquad \forall (f,\pi_1,...,\pi_k) \in \mc D^k(X).$$
Let $T$ be a metric functional of finite mass and let $\mu_T$ be the minimal finite measure as above, then $\Mass(T) \defeq \mu_T(X)$ is called the \emph{mass} of $T$ and the \emph{support} of $T$ is the closed set $\opna{supp}(T) \defeq \{x \in X \mid \mbox{$\mu_T(B_r(x))>0 \ \forall r>0\}$}$.

In the special case $(X,d_X)=(\RR^n,d_{\text{\footnotesize{Eucl}}})$ every function $\theta \in L^1(B,\RR)$ with $B \subset \RR^n$ a Borel set, induces an $n$-dimensional metric functional $[\![\theta]\!](f,\pi_1,...,\pi_n) \defeq  \int_B \theta f \det\left((\frac{\partial\pi_i}{\partial x_j})_{i,j}\right) d\lambda^n$, where $\lambda^n$ denotes the \hbox{Lebesgue} measure on $\RR^n$.

Two important constructions to produce new metric functionals from existing ones are the \emph{push-forward} and the \emph{restriction}. Let $(Y,d_Y)$ be another complete metric space and $T$ be $k$-dimensional metric functional on $(X,d_X)$. For a Lipschitz map $\varphi: (X,d_X) \to (Y,d_Y)$ the push-forward $\varphi_\#T$ is defined by $\varphi_\#T(g,\sigma_1,...,\sigma_k) = T(g\circ\varphi,\sigma_1\circ\varphi,...,\sigma_k\circ\varphi)$ for all $(g,\sigma_1,...,\sigma_k) \in \mc D^k(Y)$ and is a $k$-dimensional metric functional on $(Y,d_Y)$. For a Borel set $B\subset X$ the restriction of $T$ to $B$ is the $k$-dimensional metric functional $T_{|B}$ on $(X,d_X)$ defined by $T_{|B}(f,\pi_1,...,\pi_k)=T(f\cdot\mathbb{1}_B, \pi_1,...,\pi_k)$ for all $(f,\pi_1,...,\pi_k)\in \mc D^k(X)$.

In the following definition we denote by $\mc H^k$ the $k$-dimensional Hausdorff measure of $(X,d_X)$.

\begin{defi}[Integral current]
An \emph{integral $k$-current} on $(X,d_X)$ is a $k$-dimensional metric functional $T$ such that
\begin{enumerate}[(i)]
	\item
	$T$ is multi-linear,
	
	\item
	$T$ is continuous, i.e. if $\forall i \in \{1,...,k\}$ the sequences $(\pi_i^j)_j$ converge pointwise to $\pi_i$ for $j\to
	\infty$ and if $\sup_{i,j}(\Lip(\pi_i^j)) < \infty$ then $T(f,\pi_1^j,...,\pi_k^j) \to T(f,\pi_1,...,\pi_k) $ for $j \to
	\infty$,
	
	\item
	$T$ is local, i.e. if there are Borel sets $B_i \subset X$, $i \in \{1,...,k\}$, such that $\pi_i$ is constant on $B_i$ 
	and $\{x\in X \mid f(x) \ne 0\} \subset \bigcup_{i=1}^k B_i$, then $T(f,\pi_1,...,\pi_k)=0$,
	
	\item
	$T$ is of finite mass,
	
	\item
	$T$ is normal, i.e. $T$ and its boundary $\partial T$ satisfy the conditions $(i) - (iv)$,
	
	\item
	$T$ is integer rectifiable, i.e. $\mu_T$ is concentrated on a countably $\mc H^k$-rectifiable subset and vanishes on 
	$\mc H^k$-negligible Borel sets and for every Lipschitz map $\varphi: X \to \RR^k$ and any open subset 
	$U \subset X$ there is a function $\theta \in L^1(\RR^k,\ZZ)$ with $\varphi_\#(T_{|U})=[\![\theta]\!]$.
\end{enumerate}
\end{defi}

For $k \in \NN_0$ we denote by $\setI_k(X,d_X)$ the set of integral $k$-currents on $(X,d_X)$. It is important to mention that the push-forward, the restriction and the boundary of an integral current are again integral currents.
Further, we call an integral current $T$ \emph{closed} if it has zero-boundary $\partial T=0$.

Integral currents can be  considered as a generalisation of Lipschitz chains, which themselves are generalisations of smooth submanifolds. Indeed, every Lipschitz $k$-chain $a= \sum_i z_i \alpha_i$ with coefficients $z_i \in \ZZ$ and Lipschitz maps $\alpha_i : \Delta^k \to X$, induces an integral $k$-current $a_\# \defeq \sum_i z_i {\alpha_i}_\#([\![\mathbb{1}_{\Delta^k}]\!]) \in \setI_k(X,d_X)$.

Later we will talk about convergence of integral currents. For this we use the \emph{flat norm} of an integral $k$-current $T\in \setI_k(X,d_X)$ which is defined as 
$$\|T\|_{\text{\footnotesize{flat}}}\defeq \inf\{\Mass(R)+\Mass(S) \mid R\in \setI_k(X,d_X), S \in \setI_{k+1}(X,d_X) \text{ s.t. } R+\partial S=T\}.$$

In the case of a stratified nilpotent Lie group there is a special subclass of integral currents:

\begin{defi}[Horizontal current]
Let $G$ be a stratified nilpotent Lie group and $k\in \NN$. We call \mbox{$S\in \setI_k(G,d_R)$} \emph{horizontal} if $\opna{supp}(S)$ is almost everywhere tangent to the horizontal distribution $V_1$, i.e. \mbox{$S \in \setI_k(G,d_C)$}.
\end{defi}

For a horizontal $k$-current $S$ the mass $\Mass(S)$ is the same whatever one considers $S$ as an element of $\setI_k(G,d_R)$ or of $\setI_k(G,d_C)$.

%
%
\subsection{Isoperimetric inequalities}

We will like to refer to the optimal mass of a filling of a closed integral current $S \in \setI_k(X,d_X)$. For this we define the \emph{filling volume} of $S$ as $\FillVol(S)\defeq \inf\{\Mass(T) \mid T \in \setI_{k+1}(X,d_X),\ \partial T=S\}$.

\begin{defi}[Isoperimetric inequalities]
Let $k \in \NN$ and $\alpha>0$. 
	\begin{enumerate}[a)]
	\item
	$(X,d_X)$ \emph{admits an isoperimetric inequality of rank $\alpha$ for $\opna{\textbf{I}}_k(X,d_X)$} if there is
	a constant $C>0$ such that for every closed integral current $S \in \setI_k(X,d_X)$ 
	$$\FillVol(S) \le \begin{cases} C \cdot \Mass(S)^{\frac{k+1}{k}} \text{ for }  \Mass(S) <1,\\ C \cdot 
	\Mass(S)^\alpha \hspace{3.5mm}\text{ for }  \Mass(S)  \ge 1. \end{cases}$$
	
	\item
	$(X,d_X)$ \emph{admits an Euclidean  isoperimetric inequality for $\opna{\textbf{I}}_k(X,d_X)$} if it admits
	an isoperimetric inequality of rank $\alpha=\frac{k+1}{k}$ for $\setI_k(X,d_X)$.
	
	\item
	$(X,d_X)$ satisfies an isoperimetric inequality of \emph{strict rank} $\alpha$ for $\setI_k(X,d_X)$ if  it admits an 
	isoperimetric inequality of rank $\alpha$ for $\setI_k(X,d_X)$ and does not admit an isoperimetric inequality of
	rank $\beta$ for $\setI_k(X,d_X)$ for any $\beta < \alpha$.
	
	\item
	$(X,d_X)$ satisfies a \emph{strictly Euclidean}  isoperimetric inequality for $\setI_k(X,d_X)$ if it satisfies an
	isoperimetric inequality of strict rank $\alpha=\frac{k+1}{k}$ for $\setI_k(X,d_X)$
	\end{enumerate}
\end{defi}

As one might expect, the name Euclidean isoperimetric inequality is motivated by the fact that the Euclidean space $(\RR^n,d_{\text{\footnotesize{Eucl}}})$ satisfies for all $1\le k <n$ isoperimetric inequalities of strict rank $\alpha=\frac{k+1}{k}$ for $\setI_k(\RR^n,d_{\text{\footnotesize{Eucl}}})$.


\section{Horizontal spheres and their fillings}

\begin{lem}\label{LemApprox}
Let $(G,d_R)$ be a stratified nilpotent Lie group and let $m\in \NN$, $m< \dim(G)$, be such that $(G,d_R)$  admits for all $k\in \{1,...,m\}$  an Euclidean isoperimetric inequality for $\setI_k(G,d_R)$.
Then there is a closed horizontal current $S \in \setI_m(G,d_R)$ with $\FillVol(S)>0$.
\end{lem}

\begin{proof}
We prove the existence of $S$ in two steps: first we construct for a given Lipschitz-$m$-cycle $b$, with $\FillVol(b_\#)>0$, a sequence $(S_n)\in \setI_m(G,d_R)^\NN$ of horizontal currents such that $\opna{supp}(S_n)$ converges in the Hausdorff sense to $\opna{supp}(b_\#)$. Then we show that the sequence converges to $b_\#$ also in the flat norm and therefore there is an $n_o\in \NN$ such that $\FillVol(S_n) > 0$ for all $n \ge n_o$. We set $S\defeq S_{n_o}$ to finally prove the lemma.

Let $N\defeq \dim(\mf g)$. We fix an orthonormal (with respect to the Riemannian metric) basis $\{X_1,...,X_N\}$ of $\mf g$ and denote by $S^{N-1}\defeq\{X \in \mf g \mid \|X\|=1\}$ the unit sphere in $\mf g$. Let $\widetilde{S^m}\defeq S^{N-1} \cap \opna{span}\{X_1,...,X_{m+1}\}$ and $S^m \defeq \exp(\widetilde{S^m})$. Here we used $m<  \dim(G) =N$. Let $(\tau,\phi)$ be a Lipschitz triangulation of $S^m$ with $m$-skeleton $\tau^{(m)}$. Then the sum $b\defeq \sum_{\Delta \in \tau^{(m)}} \phi_{|\Delta}$ of the restricted maps is a Lipschitz-$m$-cycle with $\FillVol(b_\#)>0$ and $\opna{supp}(b_\#)=S^m$. Further denote by $\tau_n$, $n \in \NN$, the simplicial complex which arises from $\tau$ after $n$-fold barycentric subdivision and denote by $b_n$ the corresponding Lipschitz-$m$-cycle $\sum_{\Delta \in \tau_n^{(m)}} \phi_{|\Delta}$. We consider the maximal diameter of the simplices $\phi(\Delta)$, $\Delta \in\tau_n$, with respect to the Riemannian metric and the Carnot-Carath\'eodory metric:
\begin{align*}
\delta_R(n)\defeq& \max_{\Delta \in \tau_n} \opna{diam}_R(\phi(\Delta))=\max_{\Delta \in \tau_n}\max\{d_R(x,y) \mid x,y \in \phi(\Delta)\}\\
\delta_C(n)\defeq& \max_{\Delta \in \tau_n} \opna{diam}_C(\phi(\Delta))=\max_{\Delta \in \tau_n}\max\{d_C(x,y) \mid x,y \in \phi(\Delta)\}
\end{align*}
As the maximal diameter $\opna{diam}_n$ of simplices of $\tau_n$ converges to $0$ for $n \to \infty$, we get the same for the Riemannian diameter:
$$\delta_R(n)\le \Lip(\phi) \opna{diam}_n \to  0 \qquad (n\to \infty)$$
And using the Hölder equivalence between $(G,d_R)$ and $(G,d_C)$ (via the identity map $id_G:G \to G$ and with constant $C>0$ and exponent $\varepsilon=d^{-1}=\frac{1}{\opna{step}(G)}$, compare \cite{Gromov}) we can see the same for the Carnot-Carath\'eodory diameter:
\begin{align*}
\delta_C(n)&=\max_{\Delta \in \tau_n}\max\{d_C(x,y) \mid x,y \in \phi(\Delta)\}\le \max_{\Delta \in \tau_n}\max\{C\cdot d_R(x,y)^\varepsilon \mid x,y \in \phi(\Delta)\}\\
&= C\cdot(\max_{\Delta \in \tau_n}\max\{d_R(x,y) \mid x,y \in \phi(\Delta)\})^\varepsilon=C\cdot\delta_R(n)^\varepsilon \to 0 \qquad (n \to \infty)
\end{align*}
We construct for each $n\in \NN$ a closed horizontal  $m$-current $S_n$ with $d_{\mc H,d_R}(\opna{supp}(S_n),S^m) \to 0$ for $n \to \infty$. To do this, we introduce for each $k$-simplex $\Delta \in \tau_n$ a horizontal $k$-current $H^n_\Delta$ such that it respects the simplicial structure:
$$\partial H^n_\Delta = \sum_{\partial \Delta= \sum_i (-1)^i \Delta_i} (-1)^i H^n_{\Delta_i}$$
We start in dimension zero and step by step go up to dimension $m$.
	\begin{enumerate}
	\item[Dimension 0:]
	For each vertex $v \in \tau_n^{(0)}$ we set $H^n_v\defeq (\phi_{|v})_\#$.
	\item[Dimension 1:]
	We can connect the images of the boundary vertices $v_0,v_1$ of an $1$-simplex $\Delta \in \tau_n^{(1)}$ by a
	horizontal path $\gamma_\Delta$ of length $\le \delta_C(n)$. We set $H^n_\Delta \defeq (\gamma_\Delta)_\#$.
	As $L(\gamma_\Delta) \le \delta_C(n)$, the support of $H^n_\Delta$ stays in distance $\frac{1}{2}\delta_C(n)$ of  
	$\{\phi(v_0),\phi(v_0) \} \subset S^m$. Notice, here we used $d_C \ge d_R$. Further we have 
	$\Mass(H^n_\Delta) =\opna{length}_R(\gamma_\Delta)\le \delta_C(n)$ for all $\Delta \in \tau_n^{(1)}$.
	\item[Dimension $2\le k \le m$:]
	In this step we need the assumption that $(G,d_R)$ admits for all 
	$k\in \{1,...,m\}$ an Euclidean isoperimetric inequality for $\setI_k(G,d_R)$. 
	By \cite{Wenger11} $(G,d_C)$ therefore admits an Euclidean isoperimetric inequality  for
	compactly supported horizontal $k$-currents.
	Further we use Wenger's results on the filling radius \cite{Wenger11asymp}, which allows us to find constants 
	$D,\mu \ge 1$ such that for every closed current $T\in \setI_{k-1}(G,d_C)$ with compact support there is a filling 
	$F_T\in \setI_k(G,d_C)$ with 
		\begin{enumerate}[1)]
		\item
		$d_{\mc H,d_C}(\opna{supp}(F_T) ,\opna{supp}(T)) \le  \mu \cdot 
		\Mass(T)^\alpha$
		\ with \ $\alpha= \begin{cases}\frac{1}{k-1} \text{ if } \Mass(T)\le 1, \\   {k-1} \text{ if } \Mass(T) > 1,
		\end{cases}$ 
		\item 
		$\Mass(F_T) \le D \cdot \Mass(T)^\frac{k}{k-1}$.
		\end{enumerate}
	For $\Delta \in \tau_n^{(k)}$ with $\partial \Delta= \sum_{i=0}^{k} (-1)^i \Delta_i$ we define $H^n_\Delta$ as
	the filling $F_T\in \setI_k(G,d_C)$ of the closed compactly supported horizontal $(k-1)$-current 
	$T=\sum_{i=0}^{k} (-1)^i
	H^n_{\Delta_i} \in \setI_{k-1}(G,d_C)$.
	\end{enumerate}
We denote by $M_n^{k}$ the maximal mass of $H^n_\Delta$ for $\Delta \in \tau_n^{(k)}$. By the above construction we have 
$$M_n^{k} \le D \cdot (k+1) \cdot M_n^{k-1} \le D^k \cdot (k+1)! \cdot M_n^1 \le  D^k \cdot (k+1)! \cdot \delta_C(n).$$
By this we get the following estimate how far the support of $H^n_\Delta$, $\Delta \in \tau_n^{(k)}$, can be away from $S^m$. Denote by $D_n^k \defeq \max\{d_C(x, S^m) \mid x \in \opna{supp}(H^n_\Delta), \Delta \in \tau_n^{(k)}\}$, then:
\begin{align*}
D_n^k  &\le \mu \cdot \max_{\Delta \in \tau_n^{(k)}}\Mass(\partial H^n_\Delta)^{\alpha} +\max\{d_C(x, S^m)\mid x \in \opna{supp}(H^n_\Delta), \Delta \in \tau_n^{(k-1)}\}\\
& \le \mu \cdot \left((k+1) \cdot M_n^{k-1}\right)^{\alpha_k}+ D_n^{k-1}
\end{align*}
where $\alpha_k=k-1$ if $M_n^{k-1} > \frac{1}{k+1} $ and $\alpha_k=\frac{1}{k-1}$ otherwise.
So we obtain 
\begin{align*}
D_n^m &\le \mu \cdot \left((m+1)\cdot M_n^{m-1}\right)^{\alpha_m} + D_n^{m-1} \le \mu \cdot \sum_{i=2}^{m}\left( (i+1) \cdot M_n^{i-1}\right)^{\alpha_{i}}\\
&\le \mu \cdot \sum_{i=2}^{m} \left((i+1) \cdot D^i\cdot (i+1)! \cdot \delta_C(n)\right)^{\alpha_{i}}
\end{align*}
We set 
$$\varepsilon(n) \defeq  \mu \cdot \sum_{i=2}^{m} \left((i+1) \cdot D^i\cdot (i+1)! \cdot \delta_C(n)\right)^{\alpha_{i}}$$
and
$$S_n \defeq \sum_{\Delta \in \tau_n^{(m)}} H^n_\Delta.$$
Then $d_{\mc H,d_C}(\opna{supp}(S_n),S^m) \le \varepsilon(n)$ because by the computation above $\opna{supp}(S_n)$ is contained in the $\varepsilon(n)$-neighbourhood of $S^m$ and $S^m$ is contained in the $\delta_C(n)$-neighbourhood of the $0$-skeleton $\phi(\tau_n^{(0)})\subset \opna{supp}(S_n)$ and $\delta_C(n)\le \varepsilon(n)$.
Again by the relationship $d_R \le d_C$, the same estimate holds for the Hausdorff distance with respect to the Riemannian metric: $d_{\mc H,d_R}(\opna{supp}(S_n),S^m)\le \varepsilon(n)$.\\
As $\delta_C(n) \to 0$ for $n \to \infty$, each summand of $\varepsilon(n)$ converges to $0$ and therefore $\varepsilon(n) \to 0$ for $n \to \infty$. So the sequence $(\opna{supp}(S_n))_{n\in \NN}$ converges for $n \to \infty$ in the Hausdorff sense to $S^m$.

To see that $\|S_n-b_\#\|_{\text{\footnotesize{flat}}} \to 0$ for $n \to \infty$, we construct a sequence of integral currents $T_n\in \setI_{m+1}(G,d_R)$ with $\partial T_n= S_n -b_\#$ and $\|T_n\|_{\text{\footnotesize{flat}}} \to 0$ for $n \to \infty$. For this we will use the combinatorial (i.e. simplicial) structure of $S_n$ and again we will work our way up from dimension $1$ to dimension $m$. For the construction we need a control on the mass of $\phi$-images of $k$-simplices of $\tau_n$. Let for $k\in \{1,...,m\}$
$$\eta_k(n)\defeq \max \{\Mass((\phi_{|\Delta})_\#) \mid \Delta \in \tau_n^{(k)}\}$$
be the maximal mass of an integral $k$-current induced by a $k$-simplex of $\tau_n$. As each $k$-simplex of $\tau_n$ arises from an $n$-fold barycentric subdivision of a $k$-simplex of $\tau$ we obtain:
$$\eta_k(n) \le \Lip(\phi)^k \cdot \left(\frac{1}{(k+1)!}\right)^n \cdot \frac{\sqrt{k+1}}{k!\sqrt{2^k}} \ \to \ 0 \qquad \text{for } n\to \infty.$$
Equipped with this we can start with the construction of $T_n$.
	\begin{enumerate}
	\item[Dimension $1$:]
	Let $e \in \tau_n^{(1)}$ be an edge and let $H_e^n \in \setI_1(G,d_R)$ be the corresponding current constructed
	above.
	Then $S^n_e\defeq(H_e^n-(\phi_{|e})_\#)$ is a closed integral current of mass $\Mass(S^n_e) \le \delta_C(n)
	+ \eta_1(n)$. By the results of \cite{Wenger11asymp}, now for $(G,d_R)$, and by the Euclidean isoperimetric
	inequalities up to dimension $m$, there are constants $D',\mu' \ge 1$ such 	that we can fill 
	$S^n_e$ by an integral current $F^n_e \in \setI_2(G,d_R)$ with $\Mass(F^n_e) \le  D' \cdot \eta_1(n)^2$ and
	with $d_{\mc H,d_R}(\opna{supp}(F^n_e),\opna{supp}(S^n_e)) \le \mu' \cdot \eta_1(n)$.
	\item[Dimension $2\le k \le m$:]
	Let $\Delta \in \tau_n^{(k)}$ be a $k$-simplex with $(k-1)$-dimensional faces $\Delta_0,...,\Delta_{k}$,
	i.e. $\partial \Delta = \sum_{j=0}^{k} (-1)^j \Delta_j$, and let $F^n_{\Delta_j}, \ j \in \{0,...,k\}$, be the
	integral currents constructed in the previous step for dimension $k-1$. Then the sum 
	$$S^n_\Delta \defeq H^n_\Delta - (\phi_{|\Delta})_\# - \sum_{j=0}^{k} F^n_{\Delta_j} $$
	forms a closed integral $k$-current with $\Mass(S^n_\Delta) \le M_n^k + \eta_{k}(n) + (k+1) \cdot N_n^{k-1}$
	where $N_n^{k-1} \defeq \max\{\Mass(F^n_\Delta) \mid \Delta \in \tau_n^{(k-1)}\}$. We can fill $S^n_\Delta$,
	using again \cite{Wenger11}, with an integral current $F^n_\Delta\in\setI_{k+1}(G,d_R)$ such that 
	$\Mass(F^n_\Delta)\le D' \cdot \Mass(S^n_\Delta)^\frac{k+1} {k}$ and with 
	$d_{\mc H,d_R}(\opna{supp}(F^n_\Delta), \opna{supp}(S^n_\Delta))\le \mu' \cdot \Mass(S^n_\Delta)^\alpha$
	where $\alpha=k$ if $\Mass(S^n_\Delta) > 1$ and $\alpha=\frac{1}{k}$ otherwise.
	\end{enumerate}
We set 
$$T_n \defeq \sum_{\Delta \in \tau_n^{(m)}} F^n_\Delta \quad \in \setI_{m+1}(G,d_R)$$
and as every $\Delta \in \tau_n^{(m-1)}$ is face of exactly two $m$-simplices of $\tau_n$, we obtain by construction $\partial T_n=S_n - b_\#$. We will show that $d_{\mc H,d_R}(\opna{supp}(T_n),S^m) \to 0$ for $n\to \infty$ and therefore $\|T_n\|_{\text{\footnotesize{flat}}}\to 0$ as $\dim_{\mc H,d_R}(S^m)=m < m+1 =\dim_{\mc H,d_R}(\opna{supp}(T_n))$ (compare \cite{SormaniWenger11}).
As $S^m=\opna{supp}(b_\#) \subset \opna{supp}(T_n)$, it only remains to compute the maximal distance of points $x \in \opna{supp}(T_n)$  from $S^m$. Denote by $\sigma(n,k)\defeq \max\{d_R(x, S^m) \mid x \in \opna{supp}(F_\Delta^n),\ \Delta \in \tau_n^{(k)}\}$ the maximal distance of points in the $k$-skeleton of $\opna{supp}(T_n)$ from $S^m$. We prove by recursion $\sigma(n,k) \to 0$ for $n\to \infty$ and all $k\in \{1,...m\}$.
\begin{align*}
\sigma(n,k) & \le \max_{\Delta \in \tau_n^{(k)}} \big(d_{\mc H,d_R}(\opna{supp}(F^n_\Delta), \opna{supp}(S^n_\Delta)) + \max\{d_R(y,S^m) \mid y \in \opna{supp}(S_\Delta^n)\}\big) \\
& \le \max_{\Delta \in \tau_n^{(k)}} \big(d_{\mc H,d_R}(\opna{supp}(F^n_\Delta), \opna{supp}(S^n_\Delta))\big) + \sigma(n,k-1) + d_{\mc H, d_R}(\opna{supp}(S_n), S^m) \\
&\le \mu' \cdot (M_n^k+ \eta_k(n) + (k+1) \cdot N_n^{k-1})^{\alpha_{n,k}} + \sigma(n,k-1)+\varepsilon(n)
\end{align*}
where we have $\alpha_{n,k}=k$ if $M_n^k+ \eta_k(n) + (k+1) \cdot N_n^{k-1}\ge 1$ and $\alpha_{n,k}=\frac{1}{k}$ elsewise.
We have already seen that $M_n^k \to 0$, $\eta_k(n) \to 0$ and $\varepsilon(n) \to 0$ for $n \to \infty$. For $N_n^k$ we have the following recursive estimate:
$$ N_n^k \le D' \cdot (M_n^k+ \eta_k(n)+ (k+1)\cdot N_n^{k-1})^{\frac{k+1}{k}} $$
Together with $N_n^1 \le D' \cdot \eta_1(n)^2$ this yields $N_n^k \to 0$ for $n \to \infty$, too. So we have for $k \in \{2,...,m\}$
$$\sigma(n,k) \to 0 \text{ for } n \to \infty \ \Longleftrightarrow \ \sigma(n,k-1) \to 0 \text{ for } n \to \infty.$$
As $\sigma(n,1) \le \mu' \cdot \eta_1(n) + \varepsilon(n) \to 0$ for $n \to \infty$ we have 
$$d_{\mc H,d_R}(\opna{supp}(T_n),S^m) = \sigma(n,m) \to 0 \quad \text{for } n \to \infty.$$
Finally we obtain 
$$\FillVol(S_n) \ge \FillVol(b_\#) - \|S_n - b_\#\|_{\text{\footnotesize{flat}}} \ge \FillVol(b_\#) -  \Mass(T_n) \to \FillVol(b_\#) > 0 \quad \text{for } n \to \infty.$$
So there is an $n_o \in \NN$ with $\FillVol(S_n) >0 $ for all $n \ge n_o$. Then $S\defeq S_{n_o}$ is a closed horizontal $m$-current with positive filling volume as desired.
\end{proof}


\begin{prop}\label{prop}
Let $(G,d_R)$ be a stratified nilpotent Lie group and let $m\in \NN$ be the lowest dimension in which $(G,d_R)$ does not satisfy a strictly Euclidean isoperimetric inequality for integral currents. Then $(G,d_R)$ does not admit an Euclidean isoperimetric inequality for $\setI_{m}(G,d_R)$.
\end{prop}

\begin{proof}
As $m$ is the lowest dimension in which $(G,d_R)$ does not satisfy a strictly Euclidean isoperimetric inequality for integral currents, we have
in particular that $(G,d_R)$ admits for all $k\in\{1,...,m-1\}$ an Euclidean isoperimetric inequality for $\setI_k(G,d_R)$. 
Assume $(G,d_R)$ admits an isoperimetric inequality of rank $\alpha < \frac{m+1}{m}$ for $\setI_m(G,d_R)$ with constant $C>0$. This implies $(G,d_R)$ admits an Euclidean isoperimetric inequality for $\setI_m(G,d_R)$, too. By Lemma \ref{LemApprox} there is a closed horizontal current $S\in \setI_m(G,d_R)$ with $\FillVol(S)=V>0$. Denote by $\ell\defeq \Mass(S)$ the mass of $S$. For every $t>\frac{1}{\sqrt[m]{\ell}}$ define the closed horizontal $m$-current $S_t\defeq (s_t)_\#(S)$ and let $T_t$ be an integral $(m+1)$-current with $\partial T_t=S_t$ and $\Mass(T_t) \le C \cdot \Mass(S_t)^\alpha=C \cdot (t^m\cdot \ell)^\alpha$. Then $\widetilde{T_t}\defeq(s_{t^{-1}})_\#(T_t)$ is a filling of $S$ with 
$$\Mass(\widetilde{T_t}) \le t^{-(m+1)}\cdot \Mass(T_t) \le t^{-(m+1)} \cdot C \cdot (t^m\cdot \ell)^\alpha=t^{m\alpha-(m+1)} \cdot C \cdot \ell^\alpha.$$
As we have $\alpha < \frac{m+1}{m}$, the exponent $\epsilon=m\alpha-(m+1)$ is negative. Therefore, with $t> \max\{\left( \frac{V}{C \cdot \ell^\alpha}\right)^{\frac{1}{\epsilon}},\frac{1}{\sqrt[m]{\ell}} \}$, we obtain a filling $\widetilde{T_t}$ of $S$ with $\Mass(\widetilde{T_t})<V=\FillVol(S)$. But this is a contradiction! So $(G,d_R)$ does not satisfy an isoperimetric inequality of rank $\alpha$ for $\setI_m(G,d_R)$ for any $\alpha < \frac{m+1}{m}$.
As $(G,d_R)$ does not satisfy a strictly Euclidean isoperimetric inequality for $\setI_m(G,d_R)$ this 
yields $(G,d_R)$ cannot admit an Euclidean isoperimetric inequality for $\setI_m(G,d_R)$.
\end{proof}

\begin{rem}
The proof  shows in addition, that the dimension $k_o$ of the first non-Euclidean isoperimetric inequality in Theorem \ref{Thm} can be characterised as the smallest dimension $k$ such that there is a $k$-dimensional horizontal boundary without a horizontal filling. 
\end{rem}

\bibliography{bib}
\bibliographystyle{alpha}

\quad\\
\textsc{Courant Institute of Mathematical Sciences, New York University, New York, USA}\\
\textit{E-mail address:} gruber.moritz@outlook.de

%
%

\end{document}